\documentclass[a4paper,11pt, reqno]{amsart}
\usepackage{amssymb,amsthm,amsmath}
\usepackage{cite}

\usepackage{amsmath,amsthm,amscd,amssymb}
\usepackage{latexsym}
\usepackage[colorlinks,citecolor=red,pagebackref,hypertexnames=false]{hyperref}

\theoremstyle{plain}
\newtheorem{theorem}{Theorem}[section]
\newtheorem{Def}[theorem]{Definition}
\newtheorem{lemma}[theorem]{Lemma}

\newtheorem{proposition}[theorem]{Proposition}

\theoremstyle{definition}
\newtheorem{definition}[theorem]{Definition}

\theoremstyle{remark}

\newtheorem{case[theorem]}{Case}

\def \R{{\mathbb R}}

\def \C{{\mathbb C}}

\def\H{{\mathbb H}}

\def\norm#1.#2.{\lVert#1\rVert_{#2}}

\def\R{\mathbb R}

\def \H{{\mathcal H}}

\title{Weighted Norm Inequalities for the Opdam--Cherednik Transform}

\author{Shyam Swarup Mondal} 
\address{Department of Mathematics, Indian Institute of Technology Delhi, New Delhi 110016, India} 
\email{mondalshyam055@gmail.com}

\author{Anirudha Poria}
\address{Department of Mathematics, Indian Institute of Technology Madras, Chennai 600036, India}
\address{Department of Mathematics, Bar-Ilan University, Ramat-Gan 5290002, Israel}
\email{anirudhamath@gmail.com}

\thanks{Research supported by ERC Starting Grant No. 713927.}



\date{\today}
\begin{document}
\maketitle

\begin{abstract} 
In this paper, we study several weighted norm inequalities for the Opdam--Cherednik transform. We establish different versions of the Heisenberg--Pauli--Weyl inequality for this transform. In particular, we give an extension of this inequality using weights with different exponents and present a variation of the inequality that incorporates $L^p$-norms for the Opdam--Cherednik transform. Also, we prove a version of the Hardy--Littlewood inequality for this transform.  Finally, we give other variations of the Heisenberg--Pauli--Weyl inequality such as the Nash-type and Clarkson-type inequalities for the Opdam--Cherednik transform. 
\end{abstract}

\section{Introduction}  
Uncertainty principles have long been a mainstay of mathematical physics and classical Fourier analysis and have been studied from different points of view by several authors  \cite{fol97, hav94}.  The classical uncertainty principle is an essential restriction in harmonic analysis and it states  that a non-zero function and its Fourier transform cannot be  simultaneously sharply localized, i.e.,     it is impossible for a non-zero function and its Fourier transform to be simultaneously small. Depending on various ways of measuring the localization of a function, one can obtain different interpretations of the uncertainty principle in different contexts. Some classical uncertainty principles for the Fourier transform are   Hardy \cite{har33}, Cowling and Price \cite{cow83}, Morgan \cite{mor34},  and  Beurling \cite{hor1991} theorems. Uncertainty inequalities are some special class of uncertainty principles that give us information about how a function and its Fourier transform relate and have got considerable importance in signal analysis, physics, optics, and many other well-known areas  \cite{ric14,gro01,dono89,bial85}.  A well-known example of uncertainty inequality is the Heisenberg--Pauli--Weyl (HPW) inequality. The remarkable Heisenberg uncertainty principle \cite{heisenberg} was generalized by Weyl in \cite{weyl28}  and attributed the result to Pauli. The HPW inequality states that  for every  $f\in L^2(\R^n)$, we have
$$\left(\int_{\mathbb{R}^{n}} x_{j}^{2}|f(x)|^{2} \;d x\right)\left(\int_{\mathbb{R}^{n}} \lambda_{j}^{2}|\widehat{f}(\lambda)|^{2} \;d \lambda \right) \geq \frac{1}{4}\left(\int_{\mathbb{R}^{n}}|f(x)|^{2} \;d x\right)^{2}, \quad j \in\{1, \ldots, n\} .$$   Another variation  of the   HPW inequality states that   for  every $f \in L^{2}(\mathbb{R}^{n})$ and $\alpha>0$, there exists a positive constant $C_\alpha$ such that 
$$
\| |\cdot |^{\alpha} f  \|_{2}\; \||\cdot|^{\alpha} \widehat{f} \|_{2}\geq C_\alpha\|f\|_{2}^{2}.
$$
A strong additive version of this inequality for the classical Fourier transform on $\mathbb{R}^n$ was proved by Cowling and Price in \cite{cow83} and it states that for  all  $p, q \in[1, \infty]$ and for   all tempered functions $f$ for which $\widehat{f}$ is a function, there exists a constant $C>0$ such that
 $$
  \| |\cdot |^{a} f  \|_{p}+ \||\cdot|^{b} \widehat{f} \|_{q}\geq  C \|f\|_{2}^{2} , \quad a, b>0
 $$
 if and only if
 $
 a>\frac{1}{2}-\frac{1}{p} $ and $ b>\frac{1}{2}-\frac{1}{q}.
 $ 

Considerable attention has been devoted to discovering generalizations to new contexts for the HPW inequality and its weighted variations for various generalized transforms by several researchers.  For instance, these uncertainty inequalities were obtained in \cite{ ras2} for the classical Fourier transform,  in \cite{omrii}   for the Fourier transform associated with the Riemann–Liouville operator, in \cite{mse}  for the Fourier transform associated with the spherical mean operator, in \cite{john16} for the Heckman--Opdam transform, in \cite{ros2, shim} for the Dunkl transform,  and in \cite{ros1} for the generalized Hankel transform. Further, Wiener in \cite{win33} established the  HPW inequality  using the Hermite polynomials  for the classical Fourier transform.
Building on the ideas of Ciatti--Ricci--Sundari \cite{cia},  Soltani in \cite{soltani} proved a  more general form of the HPW  inequality for the Dunkl transform using weights with different exponents. Moreover,  Johansen in \cite{joha2016} proved several  weighted inequalities and uncertainty principles for the $(k, a)$-generalized Fourier transform.
In this paper,  we study  the  HPW inequality and its several weighted variations for the Opdam--Cherednik transform.

  Recently, various uncertainty principles  were investigated    for the Opdam--Cherednik transform.   Daher et al. in \cite{daher12} studied some  qualitative uncertainty principles for the Cherednik transform as a generalization of  the Euclidean uncertainty principles for the Fourier transform.   Mejjaoli in \cite{ majjoli14}  obtained  some   qualitative uncertainty principles for the Opdam--Cherednik transform.  These results  were  further extended to modulation spaces by  the second author in \cite{por21}.  Moreover, the  Benedicks-type uncertainty principle  for  the Opdam--Cherednik transform was  investigated by Achak and Daher  in \cite{dah18}.  In this paper, we establish several weighted norm inequalities for the Opdam--Cherednik transform. We prove different versions of the HPW inequality for this transform. In particular, we give an extension of this inequality using weights with different exponents and present a variation of the inequality that incorporates $L^p$-norms for the Opdam--Cherednik transform. Also, we study a version of the Hardy--Littlewood inequality for this transform.  Finally, we give other variations of the HPW inequality such as the Nash-type and Clarkson-type inequalities for the Opdam--Cherednik transform. Mainly, we establish the following uncertainty inequalities for the Opdam--Cherednik transform. First, we prove that  for   every $f\in L^2(\R, A_{\alpha, \beta})$,      there exists a constant $ C( \alpha, \beta  )>0$ such that 
  	$$\|| \cdot | f\|_{L^2(\R, A_{\alpha, \beta})} \; \||\cdot|\H_{\alpha, \beta} f\|_{L^2(\R, \sigma_{\alpha, \beta})} \geq  C( \alpha, \beta   ) \| f\|_{L^2(\R, A_{\alpha, \beta})}^2.
  	$$
Then,    we give  another variation of the HPW  inequality using weights with different exponents as follows: 
  	for every  $f\in L^2(\R, A_{\alpha, \beta})$ and $a, b\geq 1$,  there exists a constant $ C( \alpha, \beta  )>0$ such that 
  	$$\||\cdot|^a f\|_{L^2(\R, A_{\alpha, \beta})}^{\frac{b}{a+b}} \; \||\cdot|^b\H_{\alpha, \beta} f\|_{L^2(\R, \sigma_{\alpha, \beta})}^{\frac{a}{a+b}} \geq  C( \alpha, \beta   )^{\frac{a b}{a+b}}\| f\|_{L^2(\R, A_{\alpha, \beta})}.
  	$$
 Moreover, using the heat kernel decay estimates, we obtain a version  of the HPW  inequality which incorporates $L^p$-norms and it states as follows:	let   $p\in (1, 2], b>0, 0<a<\frac{1}{q}$,   where  $q$ is the conjugate exponent of  $p$. Then  for every $f\in L^p(\R, A_{\alpha, \beta})$  and any $t>1$, there exists a constant   $C(  a, b)>0 $ such that for all $b\leq 2$, we have 
    	$$
    	\left\|   \H_{\alpha, \beta} f \right\|_{L^q(\R, \sigma_{\alpha, \beta})} \leq  
    	C( a, b)\;t_0^{\frac{ 1}{2q}+1}\; e^{\frac{2\rho t_0^{\frac{1}{2}}}{q}}   \||\cdot |^a f\|_{L^p(\R, A_{\alpha, \beta})}^{\frac{b}{a+b}} \; \||\cdot|^b\H_{\alpha, \beta} f\|_{L^q(\R, \sigma_{\alpha, \beta})}^{\frac{a}{a+b}},
    	$$
    	and for all $b>2$, we have 
    	\begin{align*}
    		\left\|   \H_{\alpha, \beta} f \right\|_{L^q(\R, \sigma_{\alpha, \beta})} &\leq  
    		C(  a, 1)^{\frac{ b(a+1)}{a+b} } \left[ b(b-1)^{\frac{1}{b}-1} \right]^{\frac{ab}{a+b}}\bigg( t_1^{\frac{1}{2q}+1} e^{\frac{2\rho t_1^{\frac{1}{2}}}{q}}\bigg)^{\frac{ b(a+1)}{a+b} }  \\&\qquad \times \||\cdot|^a f\|_{L^p(\R, A_{\alpha, \beta})}^{\frac{b}{a+b}} \; \||\cdot|^b\H_{\alpha, \beta} f\|_{L^q(\R, \sigma_{\alpha, \beta})}^{\frac{a}{a+b}},
    	\end{align*}
    	where $t_0=t_0(a, b)=\left(\frac{a}{b}\right)^{\frac{2}{a+b}}\big(N\frac{\left\||\cdot|^{a} f\right\|_{L^p(\R, A{\alpha, \beta})}}{\||\cdot |^{b} \H_{\alpha, \beta} f  \|_{L^q(\R, \sigma_{\alpha, \beta})}}\big)^{\frac{2}{a+b}}, N\in \mathbb{N}$ and $t_1=t_0(a, 1)$.
Further, we give   some other  variations of the HPW  inequality that  involves a mixed  of  $L^1$ and $L^2$ norms  estimate.  In particular, we prove the Nash-type and Clarkson-type inequalities for the Opdam--Cherednik transform. The Nash-type inequality states as follows:   for every $f \in L^1(\R, A_{\alpha, \beta})\cap L^2(\R, A_{\alpha, \beta})$ and $s > 0$, there exists a constant $C>0$ such that 
	$$
	\left\|   \H_{\alpha, \beta} f \right\|_{L^2(\R, \sigma_{\alpha, \beta})}^2	\leq C  \left\| f\right\|_{L^1 (\R, A_{\alpha, \beta})}^{\frac{4s}{2\alpha+3+2s}}	\left\|   |\cdot|^s\H_{\alpha, \beta} f \right\|_{L^2(\R, \sigma_{\alpha, \beta})}^{\frac{2(2\alpha+3)}{2\alpha+3+2s}}.
	$$
	Finally, we prove the Clarkson-type inequality as follows:    for every $f \in L^1(\R, A_{\alpha, \beta})\cap L^2(\R, A_{\alpha, \beta})$ and $s>0$, there exists a constant $C>0$ and  $N\in \mathbb{N}$ such that  
	$$	\left\|     f \right\|_{L^1(\R, A_{\alpha, \beta})} 
	\leq C \exp\left\{ {\rho \left(N\frac{	\left\|   |\cdot|^{2s}  f \right\|_{L^1(\R, A_{\alpha, \beta})}}{\left\| f\right\|_{L^2 (\R, A_{\alpha, \beta})}}\right)^{\frac{2}{1+4s}}}\right\}  \left\| f\right\|_{L^2 (\R, A_{\alpha, \beta})}^{\frac{4s}{1+4s}}	\left\|   |\cdot|^{2s}  f \right\|_{L^1(\R, A_{\alpha, \beta})}^{\frac{1}{1+4s}}.$$

The paper is organized as follows. In Section \ref{sec2}, we present some preliminaries related to the Opdam--Cherednik transform. Also, we recall some basic definitions and properties of the weak type $L^p$-space for $\sigma$-finite measure spaces. In Section \ref{sec3}, we prove several weighted norm inequalities for the Opdam--Cherednik transform. First, we establish a version of the Hardy--Littlewood inequality for this transform. Then, we prove a version of the HPW  inequality for the Opdam--Cherednik transform and give an extension of this inequality using weights with different exponents. Also, we present another variation of the HPW inequality for this transform, which incorporates $L^p$-norms. Finally, we study other variations of the HPW inequality. In particular, we give the Nash-type and Clarkson-type inequalities for the Opdam--Cherednik transform.
\section{Preliminaries}\label{sec2}
\subsection{Harmonic analysis and the Opdam--Cherednik transform}

In this subsection, we collect the necessary definitions and results from the harmonic analysis related to the Opdam--Cherednik transform. The main references for this subsection are \cite{and15, mej14, opd95, opd00, sch08}. However, we will use the same notation as in \cite{por21}.

Let $T_{\alpha, \beta}$ denote the Jacobi--Cherednik differential--difference operator (also called the Dunkl--Cherednik operator)
\[T_{\alpha, \beta} f(x)=\frac{d}{dx} f(x)+ \Big[ 
(2\alpha + 1) \coth x + (2\beta + 1) \tanh x \Big] \frac{f(x)-f(-x)}{2} - \rho f(-x), \]
where $\alpha, \beta$ are two parameters satisfying $\alpha \geq \beta \geq -\frac{1}{2}$ and $\alpha > -\frac{1}{2}$, and $\rho= \alpha + \beta + 1$. Let $\lambda \in \C$. The Opdam hypergeometric functions $G^{\alpha, \beta}_\lambda$ on $\R$ are eigenfunctions $T_{\alpha, \beta} G^{\alpha, \beta}_\lambda(x)=i \lambda  G^{\alpha, \beta}_\lambda(x)$ of $T_{\alpha, \beta}$ that are normalized such that $G^{\alpha, \beta}_\lambda(0)=1$. The eigenfunction $G^{\alpha, \beta}_\lambda$ is given by
\[G^{\alpha, \beta}_\lambda (x)= \varphi^{\alpha, \beta}_\lambda (x) - \frac{1}{\rho - i \lambda} \frac{d}{dx}\varphi^{\alpha, \beta}_\lambda (x)=\varphi^{\alpha, \beta}_\lambda (x)+ \frac{\rho+i \lambda}{4(\alpha+1)} \sinh 2x \; \varphi^{\alpha+1, \beta+1}_\lambda (x),  \]
where $\varphi^{\alpha, \beta}_\lambda (x)={}_2F_1 \left(\frac{\rho+i \lambda}{2}, \frac{\rho-i \lambda}{2} ; \alpha+1; -\sinh^2 x \right) $ is the classical Jacobi function.

For every $ \lambda \in \C$ and $x \in  \R$, the eigenfunction
$G^{\alpha, \beta}_\lambda$ satisfy
\[ |G^{\alpha, \beta}_\lambda(x)| \leq C \; e^{-\rho |x|} e^{|\text{Im} (\lambda)| |x|},\] 
where $C$ is a positive constant. Since $\rho > 0$, we have
\begin{equation}\label{eq1}
|G^{\alpha, \beta}_\lambda(x)| \leq C \; e^{|\text{Im} (\lambda)| |x|}. 
\end{equation}
Let us denote by $C_c (\R)$ the space of continuous functions on $\R$ with compact support. The Opdam--Cherednik transform is the Fourier transform in the trigonometric Dunkl setting, and it is defined as follows.
\begin{Def}
Let $\alpha \geq \beta \geq -\frac{1}{2}$ with $\alpha > -\frac{1}{2}$. The Opdam--Cherednik transform $\mathcal{H}_{\alpha, \beta} (f)$ of a function $f \in C_c(\R)$ is defined by
\[ \H_{\alpha, \beta} (f) (\lambda)=\int_{\R} f(x)\; G^{\alpha, \beta}_\lambda(-x)\; A_{\alpha, \beta} (x) dx \quad \text{for all } \lambda \in \C, \] 
where 
\begin{align}\label{eq67}
	A_{\alpha, \beta} (x)= (\sinh |x| )^{2 \alpha+1} (\cosh |x| )^{2 \beta+1}.
\end{align}
 The inverse Opdam--Cherednik transform for a suitable function $g$ on $\R$ is given by
\[ \H_{\alpha, \beta}^{-1} (g) (x)= \int_{\R} g(\lambda)\; G^{\alpha, \beta}_\lambda(x)\; d\sigma_{\alpha, \beta}(\lambda) \quad \text{for all } x \in \R, \]
where $$d\sigma_{\alpha, \beta}(\lambda)= \left(1- \dfrac{\rho}{i \lambda} \right) \dfrac{d \lambda}{8 \pi |C_{\alpha, \beta}(\lambda)|^2}$$ and 
$$C_{\alpha, \beta}(\lambda)= \dfrac{2^{\rho - i \lambda} \Gamma(\alpha+1) \Gamma(i \lambda)}{\Gamma \left(\frac{\rho + i \lambda}{2}\right)\; \Gamma\left(\frac{\alpha - \beta+1+i \lambda}{2}\right)}, \quad \lambda \in \C \setminus i \mathbb{N}.$$
\end{Def}
We have the following estimates for $C_{\alpha, \beta}$ (see \cite{majjoli14}). There exists $N > 0$ such that for all $\lambda \in \R$ with $|\lambda| \geq N$ and for constants $k_1, k_2 > 0$, we have   
	\begin{align}\label{eq60}
		k_1 |\lambda|^{2 \alpha +2}\leq |C_{\alpha, \beta} (\lambda)|^{-2} \leq  k_2 |\lambda|^{2 \alpha +2}.
	\end{align}

The Plancherel formula is given by 
\begin{equation}\label{eq03}
\int_{\R} |f(x)|^2 A_{\alpha, \beta}(x) dx=\int_\R \H_{\alpha, \beta} (f)(\lambda) \overline{\H_{\alpha, \beta} ( \check{f})(-\lambda)} \; d \sigma_{\alpha, \beta} (\lambda),
\end{equation}
where $\check{f}(x):=f(-x)$.

Let $L^p(\R,A_{\alpha, \beta} )$ (resp. $L^p(\R, \sigma_{\alpha, \beta} )$), $p \in [1, \infty] $, denote the $L^p$-spaces corresponding to the measure $A_{\alpha, \beta}(x) dx$ (resp. $d | \sigma_{\alpha, \beta} |(x)$).   Let $p \in[1,2)$  and $p'$ is the conjugate exponent of $p$. Then there exists a constant  $C_p>0$ such that
\begin{align}\label{eq14}
	\left\|\H_{\alpha, \beta}f\right\|_{L^{p'}(\R,\sigma_{\alpha, \beta} )} \leq C_p\|f\|_{L^{p}(\R,A_{\alpha, \beta} )},
\end{align}
	  for all $f \in L^{p}(\R,A_{\alpha, \beta} )$ (see \cite{john16}). In particular,  if $f \in L^{1}(\R,A_{\alpha, \beta} )$, then 
\begin{align}\label{eq66}
\left\|\H_{\alpha, \beta}f\right\|_{L^{\infty}(\R,\sigma_{\alpha, \beta} )} \leq  \|f\|_{L^{1}(\R,A_{\alpha, \beta} )}.
\end{align}
Let $t > 0$. The heat kernel $E^{\alpha, \beta}_t$ associated with the Jacobi--Cherednik operator is defined by
\begin{equation}\label{eq04}
E^{\alpha, \beta}_t(x)=\H_{\alpha, \beta}^{-1}(e^{-t \lambda^2})(x) \quad \text{for all } x \in \R.
\end{equation}
For all $t > 0$, $E^{\alpha, \beta}_t$ is an $C^\infty$-function on $\R$. Moreover, for all $t > 0$ and all $\lambda \in \R$, we have
\begin{equation}\label{eq05}
\H_{\alpha, \beta} (E^{\alpha, \beta}_t) (\lambda)=e^{-t \lambda^2}.
\end{equation}

	Let $r>1$  and $B_{r}=\{x  \in \R  : ~| x|\leq r \}$. In the following proposition, we obtain estimates for the measure of $B_r$ with respect to $A_{\alpha, \beta}$ and $\sigma_{\alpha, \beta}.$
	
	\begin{proposition} \label{eq72}\quad 
		\begin{enumerate}
			\item \label{eq50} $A_{\alpha, \beta}\left(B_{r}\right)\leq 2 r e^{2\rho r},$ where $\rho> 0$.
			\item \label{eq55} $\sigma_{\alpha, \beta}\left(B_{r}\right)\leq  C~ r^{2\alpha+3}$, where $C $ is a positive constant.
		\end{enumerate}
		
	\end{proposition}
	\begin{proof}
	(1)	Since $A_{\alpha, \beta} (x)= (\sinh |x| )^{2 \alpha+1} (\cosh |x| )^{2 \beta+1}\leq e^{2(\alpha+\beta+1)|x|} =e^{2\rho |x|} $,  where $\rho=\alpha+\beta+1$, we have
		\begin{align*} 		 A_{\alpha, \beta}\left(B_{r}\right) 			=\int_{|x|\leq r}A_{\alpha, \beta}(x)~dx 			\leq  \int_{|x|\leq r} e^{2\rho |x|}~dx \leq 2 r e^{2\rho r}. 		\end{align*}
		
		(2) Using the estimates in (\ref{eq60}), we have	\begin{align*}  
		 	\sigma_{\alpha, \beta}\left(B_{r}\right)
			&=\int_{|\lambda|\leq r}    d |\sigma_{\alpha, \beta}| (\lambda)\\
			&\leq  \frac{k_2}{4\pi }\int_{0}^1  \sqrt{\lambda^2+  \rho^2}    \; {\lambda}^{2\alpha+1}\;d\lambda +\frac{k_2}{4\pi } \int_{1}^r  \sqrt{1+\frac{  \rho^2}{\lambda^2}}   \; {\lambda}^{2\alpha+2}\;d\lambda \\
			&\leq  \frac{k_2}{4\pi }\sqrt{1+  \rho^2 }   +\frac{k_2}{4\pi }\sqrt{1+  \rho^2 }     \int_{0}^r    \lambda^{2\alpha+2}\;d\lambda \\
			&\leq  \frac{k_2}{4\pi }\sqrt{1+  \rho^2 }    \left(1+\frac{1 }{ 2\alpha+3} \right)\; {r} ^{2\alpha+3}=C{r} ^{2\alpha+3},
		\end{align*}
 where $C=\frac{k_2}{4\pi }\sqrt{1+  \rho^2 }    (1+\frac{1 }{ 2\alpha+3} ).$
	\end{proof}
\noindent Let  $\gamma_{t}(\lambda)=\H_{\alpha, \beta} (E^{\alpha, \beta}_t) (\lambda)=e^ {-t |\lambda|^{2}}$ for all $t >0$.  Then, we obtain the following estimate for $\gamma_t.$
\begin{proposition}\label{eq51}
Let $q\geq 1.$ Then 	$	\|\gamma_{t}  \|_{L^q(\R, \sigma_{\alpha, \beta})} \leq \left( C+D~ t^{-(\alpha+\frac{3}{2})}\right)^{\frac{1}{q}}$, where $C $ and $D$ are   positive constants.
\end{proposition}

\begin{proof}We have
	\begin{align*}  
\|\gamma_{t}  \|_{L^q(\R, \sigma_{\alpha, \beta})}^q 
		&=\int_{\R} e^{-qt|\lambda|^2}   d |\sigma_{\alpha, \beta}| (\lambda)\\
				&\leq \frac{k_2}{4\pi}  \int_{0 }^1  e^{-qt\lambda^2} \sqrt{\lambda^2+\rho^2}   \; {\lambda} ^{2\alpha+1}\;d\lambda +\frac{k_2}{4\pi}  \int_{1}^\infty e^{-qt\lambda^2}  \sqrt{1+\frac{  \rho^2}{\lambda^2}}   \; {\lambda}^{2\alpha+2}\;d\lambda \\
&\leq \frac{k_2}{4\pi} \sqrt{1+  \rho^2 }  +\frac{k_2}{4\pi} \sqrt{1+  \rho^2 }  \int_{1}^\infty  e^{-qt\lambda^2}  \; {\lambda} ^{2\alpha+2}\;d\lambda \\
&\leq  \frac{k_2}{4\pi}\sqrt{1+  \rho^2 }  \left(1+   \int_{0}^\infty  e^{-qt\lambda^2}  \; {\lambda} ^{2\alpha+2}\;d\lambda \right)\\
&=  \frac{k_2}{4\pi} \sqrt{1+  \rho^2 } \left(1 +    \frac{\Gamma(\alpha+\frac{3}{2})}{2q^{\alpha+\frac{3}{2}}}  t^{-(\alpha+\frac{3}{2})}\right)\\&=C+D~ t^{-(\alpha+\frac{3}{2})},
	\end{align*}
where $C=\frac{k_2}{4\pi} \sqrt{1+  \rho^2 } $ and $D=\frac{k_2}{4\pi} \sqrt{1+  \rho^2 }\; \frac{\Gamma(\alpha+\frac{3}{2})}{2q^{\alpha+\frac{3}{2}}} $.
\end{proof}

	\begin{proposition}\label{eq52}
	Let $q\geq 2$ and $0<a<\frac{1}{q}$. Then 	$	\left\|  |\cdot|^{-a} \chi_{B_r}  \right\|_{L^q(\R, A_{\alpha, \beta})} \leq C~ r^{\frac{1}{q}}e^{{\frac{2\rho r}{q}}},$ where  $C  $ is a positive constant.
	\end{proposition}
	\begin{proof}
		Using   the relation  (1) of Proposition \ref{eq72}, we get 
		\begin{align*}
		 \left\|  |\cdot|^{-a} \chi_{B_r}  \right\|_{L^q(\R, A_{\alpha, \beta})}^q 
			&=\int_{|x|\leq r} |x|^{-aq}\;A_{\alpha, \beta}(x)~dx\\
			&=2\int_{0}^1  x^{-aq}\;A_{\alpha, \beta}(x)~dx+2 \int_{1}^r x^{-aq}\;A_{\alpha, \beta}(x)~dx\\
&\leq 2\;A_{\alpha, \beta}(1)~\int_{0}^1  x^{-aq}  ~dx+ 2\int_{1}^r A_{\alpha, \beta}(x)~dx\\
&\leq \frac{2\;A_{\alpha, \beta}(1)}{1-aq} +2 \int_{0}^r A_{\alpha, \beta}(x)~dx\\
&\leq \frac{2\;A_{\alpha, \beta}(1)}{1-aq} +2 r e^{2\rho r}\leq  \left( \frac{2\;A_{\alpha, \beta}(1)}{1-aq}+2\right) r e^{2\rho r}.
		\end{align*}
	Therefore, $	\left\|  |\cdot|^{-a} \chi_{B_r}  \right\|_{L^q(\R, A_{\alpha, \beta})} \leq C~ r^{\frac{1}{q}}e^{{\frac{2\rho r}{q}}},$
	where $C= \left( \frac{2\;A_{\alpha, \beta}(1)}{1-aq}+2\right)^{\frac{1}{q}}.$
	\end{proof}

\subsection{Weak type $L^p$-spaces} 
In this subsection, we   recall some basic definitions and properties of the weak type $L^p$-space  on  $\sigma$-finite  measure spaces. The main references for this subsection are  \cite{graf, joha2016,stein}.

Let $(X, \mu)$ be a $\sigma$-finite measure space, $1\leq p,q \leq \infty$ and let $f$ be a measurable function on $X$. We define the norm
$$
\|f\|_{L^{p, q}( X, \mu)}=\left\{\begin{array}{ll}
	\left(\frac{q}{p} \int_{0}^{\infty} t^{q / p-1} f^{*}(t)^{q} d t\right)^{1 / q} & \text {if}~ q<\infty, \\
	\sup _{t>0} t (\lambda_{f}(t))^{1 / p} & \text { if } q=\infty,
\end{array}\right.
$$
where $\lambda_{f}$ is the distribution function of $f$ and $f^{*}$ denotes  the non-increasing rearrangement of $f$, i.e.,
$$
\lambda_{f}(s)=\mu(\{x \in X:|f(x)|>s\}) \quad \text { and } \quad f^{*}(t)=\inf \left\{s: \lambda_{f}(s) \leq t\right\} .
$$
The Lorentz space $L^{p, q}(X, \mu)$ consists of measurable functions $f$ on $X$ for which $\|f\|_{L^{p, q}( X, \mu)}<\infty.$ When $p=q$, then   $L^{p, p}(X, \mu)=L^{p}(X, \mu)$.

	Let $ (X, \mu)$ and $(Y,\nu)$ be   $\sigma$-finite measure spaces, $1\leq p, q<\infty$, and let $T$ be a linear operator
	from $L^p(X, \mu)$ to $L^q(Y, \nu)$.  Then $T$ is said to be a weak type $(p, q)$ or     $(p, q)$-weak operator, if there exists a constant  $C > 0$ such that 
	\begin{align}\label{1001}
		\nu ( \{y\in Y : |T f (y)| > t \}) \leq  \left( C \frac{\|f\|_{L^p(X, \mu)}}{t }\right)^q,
	\end{align}
for all $t > 0$ and   all $ f\in L^p(X, \mu)$,	i.e., $T$ maps $L^p(X, \mu) $ boundedly into $L^{q, \infty}(Y, \nu)$. Moreover, we say that \(T\) is of strong type \((p, q)\) if it is bounded from \(L^{p}(X, \mu)\) to \(L^{q}(Y, \nu) .\) 

The inclusion relation between Lorentz spaces is given in the following proposition. 
\begin{proposition}\cite{graf}\label{eq68}\quad
\begin{enumerate}
	\item  \label{eq36} 	Let $1\leq p<\infty$ and $1\leq q<r\leq \infty$. Then   $L^{p,q}( X, \mu)\subset L^{p,r}( X, \mu)$ and consequently, there exists a constant  $C_{p,q,r}>0$ such that  $$\|f\|_{L^{p, r}( X, \mu)}\leq C_{p,q,r}\|f\|_{L^{p,q}( X, \mu)}.$$
	\item \label{eq35}
	For any $1\leq p<\infty$ and any  $f\in L^p(X, \mu)$, we have $\|f\|_{L^{p, \infty}( X, \mu)}\leq \|f\|_{L^{p}( X, \mu)}$.
\end{enumerate}

\end{proposition}

The following result on the pointwise product of two functions in Lorentz spaces is due to O'Neil \cite{neil}.

\begin{theorem} \label{neil}
	Let $q \in(2, \infty)$ and set $r=\frac{q}{q-2}$. If $g \in L^{q}(X,\mu)$ and $h \in L^{r, \infty}(X, \mu)$, then  $ gh\in L^{q^{\prime}, q}(X, \mu)$ with $$\|g h\|_{L^{q', q}(X, \mu)} \leq\|g\|_{L^{ q}(X, \mu)}\|h\|_{L^{r, \infty}(X, \mu)} .$$
\end{theorem}
The interpolation result between Lorentz spaces is given in the following theorem and can be found in  \cite{stein}.
\begin{theorem} \label{eq69}
	 Suppose $T$ is a subadditive operator of (restricted) weak types $\left(r_{j}, p_{j}\right), j=0,1$, with $r_{0}<r_{1}$ and $p_{0} \neq p_{1}$, then there exists a constant $B=B_{\theta}$ such that $\|Tf\|_{L^{p, q}(X, \mu) } \leq B\|f\|_{L^{r, q}(X, \mu) }$ for all $f$ belonging to the domain of $T$ and to $L^{r, q}(X, \mu)$, where $1 \leq q \leq \infty$,
	$$
	\frac{1}{p}=\frac{1-\theta}{p_{0}}+\frac{\theta}{p_{1}}, \quad \frac{1}{r}=\frac{1-\theta}{r_{0}}+\frac{\theta}{r_{1}} \quad \text { and } \quad 0<\theta<1 .
	$$
\end{theorem}

\section{Weighted Norm Inequalities for the Opdam--Cherednik transform}\label{sec3}
In this section, we prove  several weighted norm inequalities for the Opdam--Cherednik transform. First, we establish a version of the Hardy--Littlewood inequality for the Opdam--Cherednik transform. To obtain the result, we specialize the definition of Orlicz-type space, given in \cite{joha2016}, to the case of the measure $A_{\alpha, \beta} $.

\begin{definition} 
Let  $A_{\alpha, \beta}$ be a positive measure on $\R$ given in (\ref{eq67}).  A Young function   is a measurable function $\phi: \mathbb{R}  \rightarrow \mathbb{R}$ with $A_{\alpha, \beta}\left(\left\{x \in \mathbb{R} :|\phi(x)| \leq t\right\}\right) \leq  C t$
	for all $t>0$ and for some constant $C>0$. Given such a function $\phi$,   let  $L_{\phi}^{(p)} (\mathbb{R}, A_{\alpha, \beta}  ), 2<p<\infty$, denote the Orlicz-type space of measurable functions $f$ on $\mathbb{R}$ for which
	$$
	\|f\|_{L_{\phi}^{(p)}\left(\mathbb{R}, A_{\alpha, \beta} \right)}:=\left(\int_{\mathbb{R} }|f(x)|^{p}|\phi(x)|^{p-2} \;A_{\alpha, \beta}(x)dx\right)^{1 / p}<\infty,
	$$
	 i.e., $f\in L_{\phi}^{(p)}\left(\mathbb{R}, A_{\alpha, \beta} \right)$ if and only if $f \phi^{1-\frac{2}{p}}\in L^{p}\left(\mathbb{R}, A_{\alpha, \beta}\right)$.	
\end{definition}

\begin{theorem}
	Let $q>2$ and $f \in L_{\phi}^{(q)}\left(\mathbb{R},  A_{\alpha, \beta}\right)$, where $\phi$ is a Young function relative to $ A_{\alpha, \beta} .$ There exists a   constant $C_{q}>0$  such that
$$
\int_{\mathbb{R} }\left|\H_{\alpha, \beta}f (\lambda)\right|^{q} d| \sigma_{\alpha, \beta} |(\lambda) \leq C_{q}^{q}	\|f\|_{L_{\phi}^{(q)}\left(\mathbb{R}, A_{\alpha, \beta} \right)}^{q}.
$$
\end{theorem}
\begin{proof}
  Assume that $f$ is a simple function  on $\R.$  Let $T f(\lambda)=\H_{\alpha, \beta} f (\lambda), \lambda\in \R$. Then, using the relation (\ref{eq66}), we get $$\|T f\|_{L^{\infty, \infty}( \R,   \sigma_{\alpha, \beta})}=\|T f\|_{L^{ \infty}(\R,   \sigma_{\alpha, \beta})}\leq \|f\|_{L^{ 1}( \R,   A_{\alpha, \beta})}=\|f\|_{L^{ 1,1}( \R,   A_{\alpha, \beta})}.$$
Also, using   the Plancherel formula (\ref{eq03}) and  Proposition \ref{eq68},  we get 
  $$\|T f\|_{L^{2, \infty}( \R,   \sigma_{\alpha, \beta})}\leq \|T f\|_{L^{2, 2}( \R,   \sigma_{\alpha, \beta})}=\|T f\|_{L^{ 2}(\R,   \sigma_{\alpha, \beta})}=\|f\|_{L^{ 2}( \R,   A_{\alpha, \beta})}\leq \|f\|_{L^{ 2,1}( \R,   A_{\alpha, \beta})}.$$
  Therefore, from   Theorem \ref{eq69}, we obtain $\|T f\|_{L^{ q,q}(\R,   \sigma_{\alpha, \beta})}\leq\|f\|_{L^{ q',q}( \R,   A_{\alpha, \beta})}$.
Now, we  consider a function  $g(x)=f(x) \phi(x)^{1-\frac{2}{q}}.$ Then by hypothesis, $g\in L^{q}\left(\mathbb{R},  A_{\alpha, \beta}\right)$.   
Since $\phi$ is a Young function, we have 
$$
A_{\alpha, \beta}\left(\left\{x \in \mathbb{R} :|\phi(x)|^{\frac{2}{q}-1}>t\right\}\right)=A_{\alpha, \beta}\left(\left\{x \in \mathbb{R} :|\phi(x)|^{1-\frac{2}{q}}<\frac{1}{t}\right\}\right) \leq C t^{-\frac{q}{q-2}},
$$
and consequently $\phi^{\frac{2}{q}-1}\in L^{ r,\infty}( \R,   A_{\alpha, \beta})$, where $r=\frac{q}{q-2}.$ Now, using  O'Neil's theorem \ref{neil},  we obtain 
$$
\begin{aligned}
\left( \int_{\mathbb{R} }\left|\H_{\alpha, \beta} (f) (\lambda)\right|^{q} d \sigma_{\alpha, \beta} (\lambda) \right)^{\frac{1}{q}}& \leq\|f\|_{L^{ q',q}( \R,   A_{\alpha, \beta})} \leq\|g\|_{L^{ q}( \R,   A_{\alpha, \beta})}\big\|\phi^{\frac{2}{q}-1} \big\|_{L^{ r,\infty}( \R,   A_{\alpha, \beta})} \\
	& \leq C_q \left(\int_{\mathbb{R} }|f(x)|^{q}|\phi(x)|^{q-2} \;A_{\alpha, \beta}(x)dx\right)^{\frac{1}{q}}=C_q \|f\|_{L_{\phi}^{(q)}\left(\mathbb{R}, A_{\alpha, \beta} \right)}.
\end{aligned}
$$
Thus, the desired conclusion holds for simple functions. Now, using a standard density argument we extend the result for general functions in $L_{\phi}^{(q)}\left(\mathbb{R}, A_{\alpha, \beta} \right)$. This completes the proof of the theorem.
\end{proof}
Next, we obtain a version of the HPW  inequality for the Opdam--Cherednik transform. To prove the result,  we use the fact that 	the differential-difference operator $ \Delta_{\alpha, \beta}-|x|^{2}$ is an essentially self-adjoint operator on $L^2(\R, A_{\alpha, \beta})$. The operator $ \Delta_{\alpha, \beta} - |x|^{2}$ has a discrete spectrum. Also, for every $\lambda \in \R$ and every $f\in L^2(\R, A_{\alpha, \beta}),$
we have 
\begin{align}\label{eq70}
	\H_{\alpha, \beta} (\Delta_{\alpha, \beta} f)(\lambda)=- |\lambda|^2\H_{\alpha, \beta} f(\lambda).
\end{align} 
For a more detailed study on  the operator $\Delta_{\alpha, \beta}$,  we refer to   \cite{sch08,mej2016}.  We begin with the following additive inequality.
\begin{lemma}\label{eq007}
	Let $f\in L^2(\R, A_{\alpha, \beta})$, then 
	$$\|| \cdot | f\|_{L^2(\R, A_{\alpha, \beta})}^2+\||\cdot|\H_{\alpha, \beta} f\|_{L^2(\R, \sigma_{\alpha, \beta})}^2\geq  \lambda_{\min}{(|x|^2- \Delta_{\alpha, \beta} )  }\| f\|_{L^2(\R, A_{\alpha, \beta})}^2,
	$$
	where $ \lambda_{\min}{(|x|^2- \Delta_{\alpha, \beta} )  }$ is the minimum eigenvalue of the operator $|x|^2- \Delta_{\alpha, \beta} $.
\end{lemma}
\begin{proof}	Using   the Plancherel formula (\ref{eq03}) and the relation (\ref{eq70}), we get 
\begin{align*}
	\||\cdot|\H_{\alpha, \beta} f\|_{L^2(\R, \sigma_{\alpha, \beta})}^2&=\langle |\cdot |^2\H_{\alpha, \beta} f, \H_{\alpha, \beta} f\rangle_{L^2(\R, \sigma_{\alpha, \beta})}\\
	&=-\langle  \H_{\alpha, \beta} (\Delta_{\alpha, \beta}f), \H_{\alpha, \beta} f\rangle_{L^2(\R, \sigma_{\alpha, \beta})}\\
	&=-\langle   \Delta_{\alpha, \beta} f,   f\rangle_{L^2(\R, A_{\alpha, \beta})}.
\end{align*}
Therefore 
\begin{align*}
\|| \cdot | f\|_{L^2(\R, A_{\alpha, \beta})}^2+	\||\cdot|\H_{\alpha, \beta} f\|_{L^2(\R, \sigma_{\alpha, \beta})}^2
	&=\langle  (| \cdot |^2- \Delta_{\alpha, \beta})f,   f\rangle_{L^2(\R, A_{\alpha, \beta})}.
\end{align*}
Since the self-adjoint operator $|x|^2- \Delta_{\alpha, \beta}$    has only discrete spectra, we obtain
\begin{align*}
\|| \cdot | f\|_{L^2(\R, A_{\alpha, \beta})}^2+\||\cdot|\H_{\alpha, \beta} f\|_{L^2(\R, \sigma_{\alpha, \beta})}^2\geq \lambda_{\min}{(|x|^2- \Delta_{\alpha, \beta} )  } \| f\|_{L^2(\R, A_{\alpha, \beta})}^2,
\end{align*}
where $ \lambda_{\min}{(|x|^2- \Delta_{\alpha, \beta} )  }$ is the minimum eigenvalue of the operator $|x|^2- \Delta_{\alpha, \beta} $. This completes the proof.
\end{proof}

\begin{theorem}\label{eq13}
	Let $f\in L^2(\R, A_{\alpha, \beta})$. Then  there exists a constant $ C( \alpha, \beta  )>0$ such that 
	$$\|| \cdot | f\|_{L^2(\R, A_{\alpha, \beta})} \; \||\cdot|\H_{\alpha, \beta} f\|_{L^2(\R, \sigma_{\alpha, \beta})} \geq  C( \alpha, \beta   ) \| f\|_{L^2(\R, A_{\alpha, \beta})}^2.
	$$
\end{theorem}

\begin{proof}	For $c > 1$, we define  $f_c(x) := f (cx).$ Now, replacing   $f_c$ for $f$ in Lemma \ref{eq007}, we get 
	\begin{align}\label{eq07}
		\|| \cdot | f_c\|_{L^2(\R, A_{\alpha, \beta})}^2+\||\cdot|\H_{\alpha, \beta} f_c\|_{L^2(\R, \sigma_{\alpha, \beta})}^2\geq  \lambda_{\min}{(|x|^2- \Delta_{\alpha, \beta} )  } \| f_c\|_{L^2(\R, A_{\alpha, \beta})}^2.
	\end{align}
Using the Plancherel formula (\ref{eq03}) and the
relation (\ref{eq60}), we obtain 
\begin{align*}  
	\|| \cdot | f_c\|_{L^2(\R, A_{\alpha, \beta})}^2 
&=	\|\H_{\alpha, \beta} (| \cdot | f_c)\|_{L^2(\R, \sigma_{\alpha, \beta})}^2\\
&=\int_{\mathbb{R} }  \left|\H_{\alpha, \beta} (|\lambda|f(c\lambda))\right|^{2} d |\sigma_{\alpha, \beta}| (\lambda)\\
&=\frac{1}{c }\int_{\mathbb{R} }  \left|\H_{\alpha, \beta} \left(\frac{|\lambda|}{c}f( \lambda)\right) \right|^{2} \;\left|1-\frac{c \rho}{i\lambda}\right| \frac{d\lambda}{8 \pi |C_{\alpha, \beta} (\frac{\lambda}{c})|^2} \\
&\leq  \frac{k_2}{8\pi c ^{2\alpha+4}}	\int_{\mathbb{R} } \left|\H_{\alpha, \beta} (|\lambda|f (\lambda))\right|^{2} \;\sqrt{\frac{1}{c^2}+\frac{  \rho^2}{\lambda^2}}   \;| {\lambda} |^{2\alpha+2}\;d\lambda \\
&\leq   \frac{k_2}{8 \pi  c ^{2\alpha+4} }	\int_{\mathbb{R} }  \left|\H_{\alpha, \beta} (|\lambda|f(\lambda))\right|^{2} \;\sqrt{1+\frac{\rho^2}{\lambda^2}}   | {\lambda} |^{2\alpha+2}\;d\lambda \\
&\leq   \frac{k_2}{  k_1 c ^{2\alpha+4} } \int_{\mathbb{R} } \left|\H_{\alpha, \beta} (|\lambda|f (\lambda))\right|^{2} \;\left|1-\frac{ \rho}{i\lambda}\right| \frac{d\lambda}{8 \pi |C_{\alpha, \beta} ( \lambda )|^2} \\
&= \frac{k_2}{  k_1 c ^{2\alpha+4} }\| \H_{\alpha, \beta}(|\cdot|f)\|_{L^2(\R, \sigma_{\alpha, \beta})}^2= \frac{k_2}{  k_1 c ^{2\alpha+4} }\|| \cdot |f\|_{L^2(\R, A_{\alpha, \beta})}^2.
\end{align*}
Moreover,  using the relation (\ref{eq60}), we get 
\begin{align*} 
	\|| \cdot |\H_{\alpha, \beta} f_c\|_{L^2(\R, \sigma_{\alpha, \beta})}^2&=\int_{\mathbb{R} }|\lambda|^2 \left|\H_{\alpha, \beta}f (c\lambda)\right|^{2} d |\sigma_{\alpha, \beta}| (\lambda)\\
	&=\frac{1}{c^3}\int_{\mathbb{R} }|\lambda|^2 \left|\H_{\alpha, \beta} f (\lambda)\right|^{2} \;\left|1-\frac{c \rho}{i\lambda}\right| \frac{d\lambda}{8 \pi |C_{\alpha, \beta} (\frac{\lambda}{c})|^2} \\
		&\leq \frac{k_2}{8 \pi c ^{2\alpha+4} }	\int_{\mathbb{R} }|\lambda|^2 \left|\H_{\alpha, \beta} f (\lambda)\right|^{2} \;\sqrt{\frac{1}{c^2}+\frac{\rho^2}{\lambda^2}}   \;| {\lambda} |^{2\alpha+2}\;d\lambda \\
&\leq \frac{k_2}{8 \pi c ^{2\alpha+4} }	\int_{\mathbb{R} }|\lambda|^2 \left|\H_{\alpha, \beta}f (\lambda)\right|^{2} \;\sqrt{1+\frac{\rho^2}{\lambda^2}}\;   | {\lambda} |^{2\alpha+2}\;d\lambda \\
&\leq \frac{k_2}{  k_1 c ^{2\alpha+4} }\int_{\mathbb{R} }|\lambda|^2 \left|\H_{\alpha, \beta} f (\lambda)\right|^{2} \;\left|1-\frac{ \rho}{i\lambda}\right| \frac{d\lambda}{8 \pi |C_{\alpha, \beta} ( \lambda )|^2} \\
&\leq \frac{k_2 }{ k_1 c ^{2\alpha+4} }	\||\cdot|\H_{\alpha, \beta} f\|_{L^2(\R, \sigma_{\alpha, \beta})}^2.
\end{align*}
Therefore
\begin{align}\label{eq009}\nonumber
&\|| \cdot | f_c\|_{L^2(\R, A_{\alpha, \beta})}^2+\||\cdot|\H_{\alpha, \beta} f_c\|_{L^2(\R, \sigma_{\alpha, \beta})}^2\\&\leq \frac{k_2 }{ k_1 c ^{2\alpha+3} }[c^{-1} \|| \cdot | f\|_{L^2(\R, A_{\alpha, \beta})}^2+c\||\cdot|\H_{\alpha, \beta} f\|_{L^2(\R, \sigma_{\alpha, \beta})}^2].
\end{align} 
Further, using the Plancherel formula (\ref{eq03}) and  the relation (\ref{eq60}),  we get 
\begin{align} \label{eq71}\nonumber
	\|  f_c\|_{L^2(\R, A_{\alpha, \beta})}^2
	&=	\|\H_{\alpha, \beta} f_c\|_{L^2(\R, \sigma_{\alpha, \beta})}^2\\\nonumber
	&=\int_{\mathbb{R} }  \left|\H_{\alpha, \beta} f (c\lambda)\right|^{2} d |\sigma_{\alpha, \beta}| (\lambda)\\\nonumber
	&=\frac{1}{c }\int_{\mathbb{R} }  \left|\H_{\alpha, \beta}f (\lambda)\right|^{2} \;\left|1-\frac{c \rho}{i\lambda}\right| \frac{d\lambda}{8 \pi |C_{\alpha, \beta} (\frac{\lambda}{c})|^2} \\\nonumber
	&\geq \frac{k_1}{8\pi c ^{2\alpha+3}}	\int_{\mathbb{R} } \left|\H_{\alpha, \beta} f (\lambda)\right|^{2} \;\sqrt{1+\frac{c^2 \rho^2}{\lambda^2}}   \;| {\lambda} |^{2\alpha+2}\;d\lambda \\\nonumber
	&\geq  \frac{k_1}{8 \pi c ^{2\alpha+3} }	\int_{\mathbb{R} }  \left|\H_{\alpha, \beta} f (\lambda)\right|^{2} \;\sqrt{1+\frac{\rho^2}{\lambda^2}}   | {\lambda} |^{2\alpha+2}\;d\lambda \\\nonumber
&\geq  \frac{k_1}{  k_2 c ^{2\alpha+3} } \int_{\mathbb{R} } \left|\H_{\alpha, \beta}f (\lambda)\right|^{2} \;\left|1-\frac{ \rho}{i\lambda}\right| \frac{d\lambda}{8 \pi |C_{\alpha, \beta} ( \lambda )|^2} \\
	&= \frac{k_1}{  k_2 c ^{2\alpha+3} }\| \H_{\alpha, \beta} f\|_{L^2(\R, \sigma_{\alpha, \beta})}^2= \frac{k_1}{  k_2 c ^{2\alpha+3} }\|f\|_{L^2(\R, A_{\alpha, \beta})}^2.
\end{align}Therefore, from (\ref{eq07}), (\ref{eq009}) and (\ref{eq71}), we obtain
\begin{align}\label{eq09}
	 [c^{-1} \|| \cdot | f\|_{L^2(\R, A_{\alpha, \beta})}^2+c\||\cdot|\H_{\alpha, \beta} f\|_{L^2(\R, \sigma_{\alpha, \beta})}^2]\geq    \frac{k_1^2 }{  k_2^2  }  \lambda_{\min}{(|x|^2- \Delta_{\alpha, \beta} )  } \|f\|_{L^2(\R, A_{\alpha, \beta})}^2.
\end{align}
Since for any $f\neq 0$, $\frac {\|| \cdot | f\|_{L^2(\R, A_{\alpha, \beta})}}{\||\cdot|\H_{\alpha, \beta} f\|_{L^2(\R, \sigma_{\alpha, \beta})}}>0,$ by the Archimedean property   of real numbers, there exists a natural number $N$  such that $N\frac {\|| \cdot | f\|_{L^2(\R, A_{\alpha, \beta})}}{\||\cdot|\H_{\alpha, \beta} f\|_{L^2(\R, \sigma_{\alpha, \beta})}}>1$. Choosing  $c=N\frac{\||\cdot| f\|_{L^2(\R, A_{\alpha, \beta})}}{\||\cdot|\H_{\alpha, \beta} f\|_{L^2(\R, \sigma_{\alpha, \beta})}}  $, from (\ref{eq09}), we get 
$$\||\cdot| f\|_{L^2(\R, A_{\alpha, \beta})} \; \||\cdot|\H_{\alpha, \beta} f\|_{L^2(\R, \sigma_{\alpha, \beta})} \geq C( \alpha, \beta )\| f\|_{L^2(\R, A_{\alpha, \beta})}^2,
$$ where $C( \alpha, \beta )=\frac{k_1^2  N }{  k_2^2 (N^2+1) }  \lambda_{\min}{(|x|^2- \Delta_{\alpha, \beta} )  }$.
\end{proof}
In the following, we obtain an extension of the HPW  inequality for the Opdam--Cherednik transform  using weights with different exponents. 
\begin{proposition}
		Let $f\in L^2(\R, A_{\alpha, \beta})$ and $a, b\geq 1$. Then  there exists a constant $ C( \alpha, \beta  )>0$ such that 
	$$\||\cdot|^a f\|_{L^2(\R, A_{\alpha, \beta})}^{\frac{b}{a+b}} \; \||\cdot|^b\H_{\alpha, \beta} f\|_{L^2(\R, \sigma_{\alpha, \beta})}^{\frac{a}{a+b}} \geq  C( \alpha, \beta   )^{\frac{a b}{a+b}}\| f\|_{L^2(\R, A_{\alpha, \beta})}.
	$$
\end{proposition}
\begin{proof}
Let  $a>1$   and $a'$  such that  $\frac{1}{a}+\frac{1}{a^{\prime}}=1$. Then, we have 
$$
\begin{aligned}
\||\cdot|^a f\|_{L^2(\R, A_{\alpha, \beta})}^{\frac{1}{a}} \|f\|_{L^2(\R, A_{\alpha, \beta})}^{\frac{1}{a'}} &=\left(\int_{\mathbb{R} }|x|^{2a }|f(x)|^{2} A_{\alpha, \beta}(x)\;dx \right)^{\frac{1}{2a} }\left(\int_{\mathbb{R} }|f(x)|^{2} A_{\alpha, \beta}(x)\;dx\right)^{\frac{1}{2a'}} \\
	&=\||\cdot|^{2 }|f|^{\frac{2}{a}}\|_{L^a(\R, A_{\alpha, \beta})}^{\frac{1}{2}}
	 \||f|^{\frac{2}{a'}} \|_{L^{a'}(\R, A_{\alpha, \beta}) }^{\frac{1}{2}}.
\end{aligned}
$$
Further,  using  H\"older's inequality, we get 
 $$
\begin{aligned}
\||\cdot| f\|_{L^2(\R, A_{\alpha, \beta})}^2& \leq\left(\int_{\mathbb{R}}\left(|x|^{2}|f(x)|^{\frac{2}{a}}\right)^{a}A_{\alpha,\beta}(x)\;dx\right)^{\frac{1}{a}}\left(\int_{\mathbb{R}}\left(|f(x)|^{\frac{2}{a'}}\right)^{a'}A_{\alpha,\beta}(x)\;dx\right)^{\frac{1}{a'}} \\
& =\left(\int_{\mathbb{R} } |x|^{2a}|f(x)|^2   A_{\alpha, \beta}(x)\;dx\right)^{\frac{1}{a}} \left(\int_{\mathbb{R} }|f(x)|^{2}  A_{\alpha, \beta}(x)\;dx\right)^{\frac{1}{a'}},
\end{aligned}
$$
and consequently 
 \begin{align}\label{eq11}
 	\| |\cdot|^{a } f   \|_{L^2(\R, A_{\alpha, \beta})}^{\frac{1}{a}} \geq \frac{\||\cdot| f\|_{L^2(\R, A_{\alpha, \beta})}}{\|f\|_{L^2(\R, A_{\alpha, \beta})}^{1-\frac{1}{a}}}.
 \end{align}
 Also, by applying the same argument on  $\H_{\alpha, \beta} f$, we obtain 
 \begin{align}\label{eq12}
 	\| |\cdot |^{b } \H_{\alpha, \beta} f  \|_{L^2(\R, \sigma_{\alpha, \beta})}^{\frac{1}{b}} \geq \frac{\||\cdot| \H_{\alpha, \beta} f\|_{L^2(\R, \sigma_{\alpha, \beta})}}{\|\H_{\alpha, \beta} f\|_{L^2(\R, \sigma_{\alpha, \beta})}^{1-\frac{1}{b}}}.
 \end{align}
From (\ref{eq11}), (\ref{eq12}) and Theorem  \ref{eq13}, we get
$$
\begin{aligned}
\||\cdot|^a f\|_{L^2(\R, A_{\alpha, \beta})}^{\frac{b}{a+b}} \; \||\cdot|^b\H_{\alpha, \beta} f\|_{L^2(\R, \sigma_{\alpha, \beta})}^{\frac{a}{a+b}} &=\left[\||\cdot|^a f\|_{L^2(\R, A_{\alpha, \beta})}^{\frac{1}{a}} \; \||\cdot|^b\H_{\alpha, \beta} f\|_{L^2(\R, \sigma_{\alpha, \beta})}^{\frac{1}{b}} \right]^{\frac{a b}{a+b}} \\
	& \geq \left[\frac{\||\cdot| f\|_{L^2(\R, A_{\alpha, \beta})} \;\||\cdot| \H_{\alpha, \beta} f\|_{L^2(\R, \sigma_{\alpha, \beta})}}{\|f\|_{L^2(\R, A_{\alpha, \beta})}^{1-\frac{1}{a}}\;\|\H_{\alpha, \beta} f\|_{L^2(\R, \sigma_{\alpha, \beta})}^{1-\frac{1}{b}}} \right]^{\frac{a b}{a+b}} \\
	& \geq  C( \alpha, \beta   )^{\frac{a b}{a+b}}\|f\|_{L^2(\R, A_{\alpha, \beta})}^{\left(2-2+\frac{1}{a}+\frac{1}{b}\right) \frac{a b}{a+b}}\\&= C( \alpha, \beta   )^{\frac{a b}{a+b}}\|f\|_{L^2(\R, A_{\alpha, \beta})},
\end{aligned}
$$ where the constant $C( \alpha, \beta   )$ as in Theorem \ref{eq13}.  This completes the proof.
\end{proof}
Next, we give another variation on the HPW inequality for the Opdam--Cherednik transform, which incorporates $L^p$-norms. To prove the result, we use the heat kernel decay estimates. Recall that  $\gamma_{t}(\lambda)=\H_{\alpha, \beta} (E^{\alpha, \beta}_t) (\lambda)=e^ {-t |\lambda|^{2}}$ for all $t > 0$, where $E^{\alpha, \beta}_t$ is  the heat kernel given  in (\ref{eq04}).   In the following lemma, we obtain the decay estimate for the Opdam--Cherednik transform. 

\begin{lemma}\label{eq21}
Let  $p\in (1, 2]$ and $0<a<\frac{1}{q}$, where  $q$ is the conjugate exponent of  $p$. Then  for every $f\in L^p(\R, A_{\alpha, \beta})$  and any $t>1$, there exists a constant $C>0$ such that  
	$$
	\left\| \gamma_{t}  \H_{\alpha, \beta} f \right\|_{L^q(\R, \sigma_{\alpha, \beta})} \leq  Ct^{\frac{ 1}{2q}+1}\; e^{\frac{2\rho t^{\frac{1}{2}}}{q}}  t^{-\frac{a}{2}}\left\|  |\cdot|^a f  \right\|_{L^p(\R, A_{\alpha, \beta})}.
	$$
\end{lemma}
\begin{proof}
	 Without loss of generality, we assume that $\| | \cdot |^{a} f \|_{L^p(\R, A{\alpha, \beta})}<\infty$.	Let $r>1$  and $B_{r}=\{x  \in \R  : ~| x|\leq r \}$.   Then using    the relation (\ref{eq14}), we get 
	\begin{align}\label{eq19}\nonumber
		\left\| \gamma_{t}  \H_{\alpha, \beta} (f\chi_{B_r^c}) \right\|_{L^q(\R, \sigma_{\alpha, \beta})} &\leq \|\gamma_{t}  \|_{L^\infty(\R, \sigma_{\alpha, \beta})} \left\|  \H_{\alpha, \beta} (f\chi_{B_r^c}) \right\|_{L^q(\R, \sigma_{\alpha, \beta})} \\\nonumber
				&\leq  C_p\left\|  f\chi_{B_r^c} \right\|_{L^p(\R, A{\alpha, \beta})} \\
				&\leq  C_p r^{-a}\left\|  |\cdot |^a f  \right\|_{L^p(\R, A_{\alpha, \beta})} .
	\end{align}
Also, using   H\"older's inequality and the relation (\ref{eq66}), we obtain
	\begin{align*}
	\left\| \gamma_{t}  \H_{\alpha, \beta} (f\chi_{B_r }) \right\|_{L^q(\R, \sigma_{\alpha, \beta})} &\leq \|\gamma_{t}  \|_{L^q(\R, \sigma_{\alpha, \beta})} \left\|  \H_{\alpha, \beta} (f\chi_{B_r }) \right\|_{L^\infty(\R, \sigma_{\alpha, \beta})} \\\nonumber
	&\leq \|\gamma_{t}  \|_{L^q(\R, \sigma_{\alpha, \beta})} \left\|   f\chi_{B_r} \right\|_{L^1(\R, A_{\alpha, \beta})} \\
	&\leq   \|\gamma_{t}  \|_{L^q(\R, \sigma_{\alpha, \beta})}  \left\|  |\cdot |^{-a} \chi_{B_r}  \right\|_{L^q(\R, A_{\alpha, \beta})} \left\|  |\cdot|^a f  \right\|_{L^p(\R, A_{\alpha, \beta})}.
\end{align*}
 Using   estimates of $ \|\gamma_{t}  \|_{L^q(\R, \sigma_{\alpha, \beta})} $ and   $ \left\|  |\cdot|^{-a} \chi_{B_r}  \right\|_{L^q(\R, A_{\alpha, \beta})} $ from Propositions \ref{eq51} and   \ref{eq52}, respectively,  we get 
	\begin{align} \label{eq20} 
	\left\| \gamma_{t}  \H_{\alpha, \beta} (f\chi_{B_r }) \right\|_{L^q(\R, \sigma_{\alpha, \beta})} &\leq    C  \big( 1+t^{-(\alpha+\frac{3}{2})}\big)^{\frac{1}{q}}
\left( r e^{2\rho r}\right)^{\frac{1}{q}}	 \left\|  |\cdot|^a f  \right\|_{L^p(\R, A_{\alpha, \beta})}.
\end{align}
Thus,   from  (\ref{eq19}) and (\ref{eq20}), we obtain
\begin{align*}
	\left\| \gamma_{t}  \H_{\alpha, \beta} f \right\|_{L^q(\R, \sigma_{\alpha, \beta})} &\leq  \left\| \gamma_{t}  \H_{\alpha, \beta} (f\chi_{B_r^c}) \right\|_{L^q(\R, \sigma_{\alpha, \beta})} +	\left\| \gamma_{t}  \H_{\alpha, \beta} (f\chi_{B_r}) \right\|_{L^q(\R, \sigma_{\alpha, \beta})} \\
	&\leq \left[ C_p+Cr^a  \big( 1+ t^{-(\alpha+\frac{3}{2})}\big)^{\frac{1}{q}}
	\left( r e^{2\rho r}\right)^{\frac{1}{q}}\right]r^{-a}\left\|  |\cdot|^a f  \right\|_{L^p(\R, A_{\alpha, \beta})} .
\end{align*}
Choosing  $r = t^{\frac{1}{2}}$, we get 
\begin{align*}
	\left\| \gamma_{t}  \H_{\alpha, \beta} f \right\|_{L^q(\R, \sigma_{\alpha, \beta})}
	 &=\left[ C_p+ Ct^{\frac{a}{2}}  \big( 1+ t^{-(\alpha+\frac{3}{2})}\big)^{\frac{1}{q}} \big( t^{\frac{ 1}{2}}e^{2\rho t^{\frac{1}{2}}}\big)^{\frac{1}{q}}\right]t^{-\frac{a}{2}}\left\|  |\cdot|^a f  \right\|_{L^p(\R, A_{\alpha, \beta})} \\
	 	 &\leq \left[ C_p+    C t^{\frac{ 1}{2q}+1}e^{\frac{2\rho t^{\frac{1}{2}}}{q}} \right]t^{-\frac{a}{2}}\left\|  |\cdot|^a f  \right\|_{L^p(\R, A_{\alpha, \beta})} \\
	 	 &\leq C\; t^{\frac{ 1}{2q}+1}\; e^{\frac{2\rho t^{\frac{1}{2}}}{q}}  t^{-\frac{a}{2}}\left\|  |\cdot|^a f  \right\|_{L^p(\R, A_{\alpha, \beta})}.
\end{align*}
\end{proof}

 \begin{theorem}
Let $a, b>0$. Then  under the same assumptions as in Lemma \ref{eq21}, there exists a constant $C(  a, b)>0 $ such that for all $b\leq 2$, we have 
 	$$
 	\left\|   \H_{\alpha, \beta} f \right\|_{L^q(\R, \sigma_{\alpha, \beta})} \leq  
 C( a, b)\;t_0^{\frac{ 1}{2q}+1}\; e^{\frac{2\rho t_0^{\frac{1}{2}}}{q}}   \||\cdot |^a f\|_{L^p(\R, A_{\alpha, \beta})}^{\frac{b}{a+b}} \; \||\cdot|^b\H_{\alpha, \beta} f\|_{L^q(\R, \sigma_{\alpha, \beta})}^{\frac{a}{a+b}},
 	$$
 	and for all $b>2$, we have 
 	\begin{align*}
 	\left\|   \H_{\alpha, \beta} f \right\|_{L^q(\R, \sigma_{\alpha, \beta})} &\leq  
 	 C(  a, 1)^{\frac{ b(a+1)}{a+b} } \left[ b(b-1)^{\frac{1}{b}-1} \right]^{\frac{ab}{a+b}}\bigg( t_1^{\frac{1}{2q}+1} e^{\frac{2\rho t_1^{\frac{1}{2}}}{q}}\bigg)^{\frac{ b(a+1)}{a+b} }  \\&\qquad \times \||\cdot|^a f\|_{L^p(\R, A_{\alpha, \beta})}^{\frac{b}{a+b}} \; \||\cdot|^b\H_{\alpha, \beta} f\|_{L^q(\R, \sigma_{\alpha, \beta})}^{\frac{a}{a+b}},
 	\end{align*}
 	where $t_0=t_0(a, b)=\left(\frac{a}{b}\right)^{\frac{2}{a+b}}\big(N\frac{\left\||\cdot|^{a} f\right\|_{L^p(\R, A{\alpha, \beta})}}{\||\cdot |^{b} \H_{\alpha, \beta} f  \|_{L^q(\R, \sigma_{\alpha, \beta})}}\big)^{\frac{2}{a+b}}, N\in \mathbb{N}$ and $t_1=t_0(a, 1)$.
 \end{theorem}
   \begin{proof}
For a fixed $p \in  (1, 2]$, we assume that $f\in L^p(\R, A_{\alpha, \beta})$  satisfy $ \left\||\cdot|^{a} f\right\|_{L^p(\R, A{\alpha, \beta})}+\||\cdot |^{b} \H_{\alpha, \beta} f  \|_{L^q(\R, \sigma_{\alpha, \beta})}<\infty.$  For all $t>1$,  by Lemma \ref{eq21},  we have
   	
   	\begin{align}\label{eq23}\nonumber
   		\left\|  \H_{\alpha, \beta} f \right\|_{L^q(\R, \sigma_{\alpha, \beta})} &\leq  \left\| \gamma_{t}  \H_{\alpha, \beta} f \right\|_{L^q(\R, \sigma_{\alpha, \beta})} +	\left\|(1-\gamma_{t} )   \H_{\alpha, \beta}f \right\|_{L^q(\R, \sigma_{\alpha, \beta})} \\
   		&\leq C t^{\frac{1}{2q}+1}\; e^{\frac{2\rho t^{\frac{1}{2}}}{q}}  t^{-\frac{a}{2}}\left\|  |\cdot|^a f  \right\|_{L^p(\R, A_{\alpha, \beta})}+	\left\|(1-\gamma_{t} )   \H_{\alpha, \beta}f \right\|_{L^q(\R, \sigma_{\alpha, \beta})} .
   	\end{align}
Moreover, 
\begin{align}\label{eq61}\nonumber
	\left\| (1-\gamma_{t} )   \H_{\alpha, \beta}f \right\|_{L^q(\R, \sigma_{\alpha, \beta})}  &=t^{\frac{b}{2}} \big\|\left(t |\cdot|^{2}\right)^{-\frac{b}{2}} (1-\gamma_{t}) |\cdot |^{b} \H_{\alpha, \beta} f \big\|_{L^q(\R, \sigma_{\alpha, \beta})}\\\nonumber
	&\leq t^{\frac{b}{2}} \big\|\left(t |\cdot|^{2}\right)^{-\frac{b}{2}} (1-\gamma_{t})   \big\|_{L^\infty (\R, \sigma_{\alpha, \beta})} \||\cdot |^{b} \H_{\alpha, \beta} f  \|_{L^q(\R, \sigma_{\alpha, \beta})}\\
	&  \leq C_1 \;t^{\frac{b}{2}} \; \||\cdot |^{b} \H_{\alpha, \beta} f  \|_{L^q(\R, \sigma_{\alpha, \beta})},
\end{align}
  whenever  $b\leq 2.$     Thus, from (\ref{eq23}) and (\ref{eq61}), we have 
\begin{align}\label{eq24}\nonumber
\left\|  \H_{\alpha, \beta} f \right\|_{L^q(\R, \sigma_{\alpha, \beta})} 
&\leq  C \;t^{\frac{1}{2q}+1}\; e^{\frac{2\rho t^{\frac{1}{2}}}{q}}  t^{-\frac{a}{2}}\left\|  |\cdot|^a f  \right\|_{L^p(\R, A_{\alpha, \beta})}+C_1 \;t^{\frac{b}{2}} \||\cdot |^{b} \H_{\alpha, \beta} f  \|_{L^q(\R, \sigma_{\alpha, \beta})}\\
&\leq  C_2\; t^{\frac{ 1}{2q}+1}\; e^{\frac{2\rho t^{\frac{1}{2}}}{q}}   
\left[   t^{-\frac{a}{2}}\| | \cdot|^{a} f \|_{L^p(\R, A{\alpha, \beta})}+  t^{\frac{b}{2}} \||\cdot |^{b} \H_{\alpha, \beta} f  \|_{L^q(\R, \sigma_{\alpha, \beta})}\right].
\end{align}
Since for any $f\neq 0$, $\frac{a\left\||\cdot|^{a} f\right\|_{L^p(\R, A{\alpha, \beta})}}{b\||\cdot |^{b} \H_{\alpha, \beta} f  \|_{L^q(\R, \sigma_{\alpha, \beta})}}>0$, there exists a natural number $N$  such that $N\frac{a\left\||\cdot|^{a} f\right\|_{L^p(\R, A{\alpha, \beta})}}{b\||\cdot |^{b} \H_{\alpha, \beta} f  \|_{L^q(\R, \sigma_{\alpha, \beta})}}>1$.
We consider the function $g$   defined on $[1, \infty)$ by
   $$
   g(t)=   t^{-\frac{a}{2}}\| | \cdot|^{a} f \|_{L^p(\R, A_{\alpha, \beta})}+  t^{\frac{b}{2}} \||\cdot |^{b} \H_{\alpha, \beta} f  \|_{L^q(\R, \sigma_{\alpha, \beta})}.
   $$
   Then  the minimum of the function $g$ is attained at the point
   \begin{align}\label{eq56}
  t_0=\left(\frac{a}{b}\right)^{\frac{2}{a+b}}\bigg(N\frac{\left\||\cdot|^{a} f\right\|_{L^p(\R, A{\alpha, \beta})}}{\||\cdot |^{b} \H_{\alpha, \beta} f  \|_{L^q(\R, \sigma_{\alpha, \beta})}}\bigg)^{\frac{2}{a+b}}
    \end{align}
   and $$
   g(t_0)=\left[ \left(\frac{b}{aN}\right)^{\frac{a}{a+b}} + \left(\frac{aN}{b}\right)^{\frac{b}{a+b}}\right] \||\cdot|^a f\|_{L^p(\R, A_{\alpha, \beta})}^{\frac{b}{a+b}} \; \||\cdot|^b\H_{\alpha, \beta} f\|_{L^q(\R, \sigma_{\alpha, \beta})}^{\frac{a}{a+b}}.$$
   Thus, from (\ref{eq24}),   we get 
   \begin{align}\label{eq25}
   	\left\|  \H_{\alpha, \beta} f \right\|_{L^q(\R, \sigma_{\alpha, \beta})} 
   	&\leq  C(  a, b)\;t_0^{\frac{ 1}{2q}+1}\; e^{\frac{2\rho t_0^{\frac{1}{2}}}{q}}   \||\cdot|^a f\|_{L^p(\R, A_{\alpha, \beta})}^{\frac{b}{a+b}} \; \||\cdot|^b\H_{\alpha, \beta} f\|_{L^q(\R, \sigma_{\alpha, \beta})}^{\frac{a}{a+b}},
   \end{align}
where $C(  a, b)=C_2\left[ \left(\frac{b}{aN}\right)^{\frac{a}{a+b}} + \left(\frac{aN}{b}\right)^{\frac{b}{a+b}}\right]$.

Now, we consider the case   for $b>2$. Since $ u \leq 1+u^{b}$ for all $u>0$, in particular for $u$ of the form    $u=\frac{|\lambda|}{ \varepsilon}$ for all $\varepsilon>0$,   the inequality becomes $\frac{|\lambda|}{ \varepsilon} \leq 1+\left(\frac{|\lambda|}{ \varepsilon}\right)^{b}$. Then, we have
 \begin{align}\label{eq79} 
 	 \| |\cdot|\H_{\alpha, \beta} f  \|_{L^q(\R, \sigma_{\alpha, \beta})}  \leq \varepsilon \| \H_{\alpha, \beta} f  \|_{L^q(\R, \sigma_{\alpha, \beta})}  +\varepsilon^{1-b} \| |\cdot|^b\H_{\alpha, \beta} f  \|_{L^q(\R, \sigma_{\alpha, \beta})}. \end{align} 
Let 
 \begin{align}\label{eq80}  g(\varepsilon)=\varepsilon \|  \H_{\alpha, \beta} f  \|_{L^q(\R, \sigma_{\alpha, \beta})}  +\varepsilon^{1-b} \| |\cdot|^b\H_{\alpha, \beta} f  \|_{L^q(\R, \sigma_{\alpha, \beta})}. \end{align} 
  Then, the minimum of the function $g$ is attain at the point $$ \varepsilon_{0}=(b-1)^{\frac{1}{b}}\left(\frac{\| |\cdot|^b\H_{\alpha, \beta} f  \|_{L^q(\R, \sigma_{\alpha, \beta})}}{\|  \H_{\alpha, \beta} f  \|_{L^q(\R, \sigma_{\alpha, \beta})}}\right)^{\frac{1}{b}}. $$ Optimizing in $\varepsilon_{0}$, from (\ref{eq79}) and (\ref{eq80}), we get
 \begin{align}\label{eq26} 
 	\| |\cdot|\H_{\alpha, \beta} f  \|_{L^q(\R, \sigma_{\alpha, \beta})}  \leq b(b-1)^{\frac{1}{b}-1}\| \H_{\alpha, \beta} f  \|_{L^q(\R, \sigma_{\alpha, \beta})}^{1-\frac{1}{b}}~\| |\cdot|^b\H_{\alpha, \beta} f  \|_{L^q(\R, \sigma_{\alpha, \beta})} ^{\frac{1}{b}}. \end{align} 
Again,  from the relation (\ref{eq25}) with $b=1$, we get  
  \begin{align} \label{eq75}
  	\left\|  \H_{\alpha, \beta} f \right\|_{L^q(\R, \sigma_{\alpha, \beta})}  	\leq C ( a, 1)\;t_1^{\frac{ 1}{2q}+1}\; e^{\frac{2\rho t_1^{\frac{1}{2}}}{q}}\||\cdot|^a f\|_{L^p(\R, A_{\alpha, \beta})}^{\frac{1}{a+1}} \; \||\cdot|\H_{\alpha, \beta} f\|_{L^q(\R, \sigma_{\alpha, \beta})}^{\frac{a}{a+1}}, \end{align} 
   where $t_1=t_0(a, 1).$ Therefore, using the  relation (\ref{eq75}), from (\ref{eq26}), we obtain 
 $$ \begin{aligned} 	\left\|  \H_{\alpha, \beta} f \right\|_{L^q(\R, \sigma_{\alpha, \beta})} &\leq  C ( a, 1) \left[ b(b-1)^{\frac{1}{b}-1} \right]^{\frac{a}{a+1}} \;t_1^{\frac{ 1}{2q}+1}\; e^{\frac{2\rho t_1^{\frac{1}{2}}}{q}} \\&\qquad\times 	\||\cdot |^a f\|_{L^p(\R, A_{\alpha, \beta})}^{\frac{1}{a+1}}\| \H_{\alpha, \beta} f  \|_{L^q(\R, \sigma_{\alpha, \beta})}^{\frac{a(b-1)}{b(a+1)}}~\| |\cdot|^b\H_{\alpha, \beta} f  \|_{L^q(\R, \sigma_{\alpha, \beta})} ^{\frac{a}{b(a+1)}}.
 \end{aligned} $$  
 Thus
 \begin{align*} \left\|  \H_{\alpha, \beta} f \right\|_{L^q(\R, \sigma_{\alpha, \beta})} ^{\frac{a+b} { b(a+1)}} &\leq C(  a, 1) \left[ b(b-1)^{\frac{1}{b}-1} \right]^{\frac{a}{a+1}} \;t_1^{\frac{ 1}{2q}+1} e^{\frac{2\rho t_1^{\frac{1}{2}}}{q}}  \\&\qquad \times  \||\cdot|^a f\|_{L^p(\R, A_{\alpha, \beta})}^{\frac{1} {a+1}}\| |\cdot|^b\H_{\alpha, \beta} f  \|_{L^q(\R, \sigma_{\alpha, \beta})} ^{\frac{a}{ b(a+1)}}, \end{align*}
  and consequently 
  \begin{align*} \left\|   \H_{\alpha, \beta} f \right\|_{L^q(\R, \sigma_{\alpha, \beta})}& \leq  C(  a, 1)^{\frac{ b(a+1)}{a+b} } \left[ b(b-1)^{\frac{1}{b}-1} \right]^{\frac{ab}{a+b}}\bigg( t_1^{\frac{1}{2q}+1} e^{\frac{2\rho t_1^{\frac{1}{2}}}{q}}\bigg)^{\frac{ b(a+1)}{a+b} }  \\&\qquad \times \||\cdot|^a f\|_{L^p(\R, A_{\alpha, \beta})}^{\frac{b}{a+b}} \; \||\cdot|^b\H_{\alpha, \beta} f\|_{L^q(\R, \sigma_{\alpha, \beta})}^{\frac{a}{a+b}}. \end{align*}
	This completes the proof of the theorem. 
   \end{proof}
Next, we study other variations of the  HPW inequality.  In particular, we give the  Nash-type and Clarkson-type inequalities for the Opdam--Cherednik transform. These inequalities involve a mixed $L^1$ and $L^2$ norms estimate. 
  \begin{theorem}[Nash-type inequality]
   	Let $s>0$. Then for every $f \in L^1(\R, A_{\alpha, \beta})\cap L^2(\R, A_{\alpha, \beta})$, there exists a constant $C>0$ such that 
   	$$
	\left\|   \H_{\alpha, \beta} f \right\|_{L^2(\R, \sigma_{\alpha, \beta})}^2	\leq C  \left\| f\right\|_{L^1 (\R, A_{\alpha, \beta})}^{\frac{4s}{2\alpha+3+2s}}	\left\|   |\cdot|^s\H_{\alpha, \beta} f \right\|_{L^2(\R, \sigma_{\alpha, \beta})}^{\frac{2(2\alpha+3)}{2\alpha+3+2s}}.
   	$$
   \end{theorem}

\begin{proof}
	Let $f\in  L^2(\R, A_{\alpha, \beta})$,  $r>1$ and     $B_{r}=\{x  \in \R  : ~| x|\leq r \}$.  Then, using  the Plancherel formula (\ref{eq03}), we have 
	\begin{align}\label{eq28}
		\left\|   f \right\|_{L^2(\R, A_{\alpha, \beta})}^2 =	\left\|   \H_{\alpha, \beta} f \right\|_{L^2(\R, \sigma_{\alpha, \beta})}^2= \left\|   \H_{\alpha, \beta} (f\chi_{B_r}) \right\|_{L^2(\R, \sigma_{\alpha, \beta})}^2 +	\left\|   \H_{\alpha, \beta} (f\chi_{B_r^c}) \right\|_{L^2(\R, \sigma_{\alpha, \beta})}^2.
	\end{align}
Now,  from the relation (2) of Proposition \ref{eq72}, we get 
	\begin{align*} 
\left\|   \H_{\alpha, \beta} (f\chi_{B_r}) \right\|_{L^2(\R, \sigma_{\alpha, \beta})}^2=\int_{B_r  }  \left|\H_{\alpha, \beta} f(\lambda)\right|^{2} d |\sigma_{\alpha, \beta}| (\lambda)\leq C_1~\left\|\H_{\alpha, \beta}f\right\|_{L^\infty (\R, \sigma_{\alpha, \beta})}^{2} r^{2\alpha+3}
\end{align*}
and 
	\begin{align*} 
	\left\|   \H_{\alpha, \beta} (f\chi_{B_r^c}) \right\|_{L^2(\R, \sigma_{\alpha, \beta})}^2&=\int_{B_r^c }  \left|\H_{\alpha, \beta} f(\lambda)\right|^{2} d |\sigma_{\alpha, \beta}| (\lambda)\\&\leq r^{-2s}\int_{B_r^c } |\lambda|^{2s}  \left|\H_{\alpha, \beta} f(\lambda)\right|^{2} d |\sigma_{\alpha, \beta}| (\lambda)\\&\leq r^{-2s}	\left\|   |\cdot|^s\H_{\alpha, \beta} f \right\|_{L^2(\R, \sigma_{\alpha, \beta})}^2.
\end{align*}
Hence,  using the relation  (\ref{eq66}), from (\ref{eq28}), we obtain 
	\begin{align}\label{eq29} \nonumber
 	\left\|   \H_{\alpha, \beta} f \right\|_{L^2(\R, \sigma_{\alpha, \beta})}^2
&\leq C_1~ r^{2\alpha+3} \left\|\H_{\alpha, \beta}f\right\|_{L^\infty (\R, \sigma_{\alpha, \beta})}^{2}+r^{-2s}	\left\|   |\cdot|^s\H_{\alpha, \beta} f \right\|_{L^2(\R, \sigma_{\alpha, \beta})}^2\\
&\leq C_1~ r^{2\alpha+3}  \left\| f\right\|_{L^1 (\R, A_{\alpha, \beta})}^{2}+r^{-2s}	\left\|   |\cdot|^s\H_{\alpha, \beta} f \right\|_{L^2(\R, \sigma_{\alpha, \beta})}^2.
\end{align}
Since for any $f\neq 0$, $ \frac{	\left\|   |\cdot|^s\H_{\alpha, \beta} f \right\|_{L^2(\R, \sigma_{\alpha, \beta})}^2}{\left\| f\right\|_{L^1 (\R, A_{\alpha, \beta})}^{2}}>0,$ by the Archimedean property   of real numbers, there exists a natural number $N$  such that $N\frac{	\left\|   |\cdot|^s\H_{\alpha, \beta} f \right\|_{L^2(\R, \sigma_{\alpha, \beta})}^2}{\left\| f\right\|_{L^1 (\R, A_{\alpha, \beta})}^{2}}>1$. Then, the  right hand side of (\ref{eq29}) is minimized for  $r^{2\alpha+3+2s}=N\frac{	\left\|   |\cdot|^s\H_{\alpha, \beta} f \right\|_{L^2(\R, \sigma_{\alpha, \beta})}^2}{\left\| f\right\|_{L^1 (\R, A_{\alpha, \beta})}^{2}}$. Therefore 
\begin{align*} 
	\left\|   \H_{\alpha, \beta} f \right\|_{L^2(\R, \sigma_{\alpha, \beta})}^2	
	\leq C \left\| f\right\|_{L^1 (\R, A_{\alpha, \beta})}^{\frac{4s}{2\alpha+3+2s}}	\left\|   |\cdot|^s\H_{\alpha, \beta} f \right\|_{L^2(\R, \sigma_{\alpha, \beta})}^{\frac{2(2\alpha+3)}{2\alpha+3+2s}},
\end{align*}
where $C=(C_1 N+1)N^{-\frac{2s}{2\alpha+3+2s}}.$
\end{proof}

\begin{theorem}[Clarkson-type inequality]
	 Let $s>0$ and $N\in \mathbb{N}$.  Then for every $f \in L^1(\R, A_{\alpha, \beta})\cap L^2(\R, A_{\alpha, \beta})$, there exists a constant $C>0$ such that  
	 $$	\left\|     f \right\|_{L^1(\R, A_{\alpha, \beta})} 
	 \leq C \exp\left\{ {\rho \left(N\frac{	\left\|   |\cdot|^{2s}  f \right\|_{L^1(\R, A_{\alpha, \beta})}}{\left\| f\right\|_{L^2 (\R, A_{\alpha, \beta})}}\right)^{\frac{2}{1+4s}}}\right\}  \left\| f\right\|_{L^2 (\R, A_{\alpha, \beta})}^{\frac{4s}{1+4s}}	\left\|   |\cdot|^{2s}  f \right\|_{L^1(\R, A_{\alpha, \beta})}^{\frac{1}{1+4s}}.$$
\end{theorem}
  \begin{proof}
  	Let $f \in L^1(\R, A_{\alpha, \beta})\cap L^2(\R, A_{\alpha, \beta})$, $r>1$ and      $B_{r}=\{x  \in \R  : ~| x|\leq r \}$. Using H\"older's  inequality and  the relation (1) of  Proposition \ref{eq72}, we get 
  		\begin{align} \label{eq34}\nonumber
  		\left\|   f \right\|_{L^1(\R, A_{\alpha, \beta})}&=  \int_{B_r } |f(x)| A_{\alpha, \beta}(x)\;dx+\int_{B_r^c } |f(x)| A_{\alpha, \beta}(x)\;dx\\\nonumber
  		&\leq A_{\alpha, \beta}(B_r )^{\frac{1}{2}}\left\|   f \right\|_{L^2(\R, A_{\alpha, \beta})}+r^{-2s}\int_{B_r^c } |x|^{2s}|f(x)| A_{\alpha, \beta}(x)\;dx\\\nonumber
  		  		&\leq   2^{\frac{1}{2}} r^{\frac{1}{2}}e^{\rho r} \left\|   f \right\|_{L^2(\R, A_{\alpha, \beta})}+r^{-2s}\| |\cdot|^{2s}f\|_{L^1{(\R, A_{\alpha, \beta})}}\\
  		  		&\leq   e^{\rho r}( 2^{\frac{1}{2}} r^{\frac{1}{2}} \left\|   f \right\|_{L^2(\R, A_{\alpha, \beta})}+r^{-2s}\| |\cdot|^{2s}f\|_{L^1{(\R, A_{\alpha, \beta})}}).
  	\end{align}
  
  Since for any $f\neq 0$, $ \frac{	\left\|   |\cdot|^{2s}  f \right\|_{L^1(\R, A_{\alpha, \beta})}}{\left\| f\right\|_{L^2 (\R, A_{\alpha, \beta})}}>0,$ by the Archimedean property   of real numbers, there exists a natural number $N$  such that $N\frac{	\left\|   |\cdot|^{2s}  f \right\|_{L^1(\R, A_{\alpha, \beta})}}{\left\| f\right\|_{L^2 (\R, A_{\alpha, \beta})}}>1$.  Then,  the   right hand side of (\ref{eq34}) is minimized for   $r^{\frac{1}{2}+2s}=N\frac{	\left\|   |\cdot|^{2s}  f \right\|_{L^1(\R, A_{\alpha, \beta})}}{\left\| f\right\|_{L^2 (\R, A_{\alpha, \beta})}}$. Therefore 
  \begin{align*} 
  	\left\|     f \right\|_{L^1(\R, A_{\alpha, \beta})} 
  	&\leq C \exp\left\{ {\rho \left(N\frac{	\left\|   |\cdot|^{2s}  f \right\|_{L^1(\R, A_{\alpha, \beta})}}{\left\| f\right\|_{L^2 (\R, A_{\alpha, \beta})}}\right)^{\frac{2}{1+4s}}}\right\}  \left\| f\right\|_{L^2 (\R, A_{\alpha, \beta})}^{\frac{4s}{1+4s}}	\left\|   |\cdot|^{2s}  f \right\|_{L^1(\R, A_{\alpha, \beta})}^{\frac{1}{1+4s}},
  \end{align*}
where $C=( 2^{\frac{1}{2}} N+1)N^{-\frac{4s}{1+4s}}.$
  \end{proof} 
\section*{Acknowledgments}
The first author gratefully acknowledges the support provided by IIT Delhi, Government of India.  The second author   is deeply indebted to Prof. Nir Lev for several fruitful discussions and generous comments. The second author is also grateful to the Science and Engineering Research Board (SERB), Government of India for providing the National Post-Doctoral Fellowship [File No. PDF/2021/000192].

\end{document}